\documentclass[reqno]{amsart}
\usepackage[utf8]{inputenc}

\usepackage{amssymb}
\usepackage{graphicx}
\usepackage{latexsym}
\usepackage{stmaryrd}
\usepackage{enumitem}
\usepackage[bookmarks]{hyperref}
\usepackage{multirow}
\usepackage{xcolor}
\usepackage{relsize}
\usepackage{comment}
\usepackage{xifthen}
\usepackage{mathrsfs}
\usepackage{svg}
\usepackage{mathtools}
\usepackage{tikz}
\usetikzlibrary{arrows}
\usetikzlibrary{positioning}
\colorlet{darkishRed}{red!80!black}
\colorlet{darkishBlue}{blue!60!black}
\colorlet{darkishGreen}{green!60!black}
\hypersetup{
    draft = false,
    bookmarksopen=true,
    colorlinks,
    linkcolor={red!60!black},
    citecolor={green!60!black},
    urlcolor={blue!60!black}
}

\usetikzlibrary{arrows}

\newcommand{\nao}[1][]{%
\ifthenelse{\equal{#1}{}}{\trianglelefteq_T}{\trianglelefteq_{T_{#1}}\!}%
}

\def\Dv{\lowbkwd D{0.3}1}
\def\Tv{\lowbkwd T{0.3}1}
\def\Rv{\lowbkwd R{0.3}1}

\def\lowfwd #1#2#3{{\mathop{\kern0pt #1}\limits^{\kern#2pt\raise.#3ex
\vbox to 0pt{\hbox{$\scriptscriptstyle\rightarrow$}\vss}}}}
\def\lowbkwd #1#2#3{{\mathop{\kern0pt #1}\limits^{\kern#2pt\raise.#3ex
\vbox to 0pt{\hbox{$\scriptscriptstyle\leftarrow$}\vss}}}}

\def\ve{\kern-1.5pt\lowfwd e{1.5}2\kern-1pt}
\def\ev{\kern-1pt\lowbkwd e{0.5}2\kern-1pt}
\def\vf{\kern-2pt\lowfwd f{2.5}2\kern-1pt}

\newcommand{ \N } { \mathbb{N} }
\newcommand{ \Z } { \mathbb{Z} }

\makeatletter

\def\calCommandfactory#1{%
   \expandafter\def\csname c#1\endcsname{\mathcal{#1}}}
\def\frakCommandfactory#1{%
   \expandafter\def\csname frak#1\endcsname{\mathfrak{#1}}}

\newcounter{ctr}
\loop
  \stepcounter{ctr}
  \edef\X{\@Alph\c@ctr}
  \expandafter\calCommandfactory\X
  \expandafter\frakCommandfactory\X
  \edef\Y{\@alph\c@ctr}
  \expandafter\frakCommandfactory\Y
\ifnum\thectr<26
\repeat

\newcommand{\dc}[1]{\lceil #1\rceil}
\newcommand{\uc}[1]{\lfloor #1\rfloor}

\newtheorem{theorem}{Theorem}[section]

\newtheorem{mainresult}{Theorem}

\newtheorem{corollary}[theorem]{Corollary}
\newtheorem{lemma}[theorem]{Lemma}
\newtheorem{example}[theorem]{Example}

\newtheorem{innercustomthm}{Theorem}

\newenvironment{customthm}[1]
  {\renewcommand\theinnercustomthm{#1}\innercustomthm}
  {\endinnercustomthm}

\theoremstyle{definition}

\theoremstyle{remark}

\setenumerate{label={\normalfont (\roman*)},itemsep=0pt}

\title[Hamiltonicity in infinite tournaments]{\Large Hamiltonicity in infinite tournaments}

\author{Ruben Melcher}
\address{University of Hamburg, Department of Mathematics, Bundesstraße 55 (Geomatikum), 20146 Hamburg, Germany}
\email{ruben.melcher@uni-hamburg.de}

\@namedef{subjclassname@2020}{\textup{2020} Mathematics Subject Classification}

\begin{document}

\begin{abstract}

We prove that for all countable tournaments $D$ the recently discovered compactification $|D|$ by their ends and limit edges contains a topological Hamilton path: a topological arc that contains every vertex. If $D$ is strongly connected, then $|D|$ contains a topological Hamilton circle. 

These results extend well-known theorems about finite tournaments, which we show do not extend to the infinite in a purely combinatorial setting.
\end{abstract}
\vspace*{-1cm}
\maketitle

\vspace*{-.7cm}
\section{Introduction}
\noindent

\noindent A natural aim in infinite graph theory is to extend known theorems about finite graphs to infinite graphs. However the right way to do this is not always to apply the finite statement to infinite graphs verbatim: it often fails for trivial reasons, or becomes trivially true. A fruitful attempt to overcome this issue is to not only consider the graph itself, but the graph together with `points at infinity': its ends.

Formally, an \emph{end} of a graph is an equivalence class of its rays, where two rays are equivalent if no finite set of vertices separates them. A graph $G$ together with its ends naturally forms a topological space $|G|$. For locally finite $G$, this is its well-known Freudenthal compactification. The topological properties of the space $|G|$ have been extensively studied~\cite{diestel2006end, diestel2003graph,ApproximatingNormalTrees, polat1996ends, polat1996ends2, sprussel2008end}. Letting topological arcs and circles in $|G|$ take the role of paths and cycles in $G$, it often becomes possible to extend theorems about paths and cycles in finite graphs to infinite graphs. Examples include Euler's theorem~\cite{BergerBruhnEuler,GartsidePitzEuler}, arboricity and \hbox{tree-packing~\cite{DiestelBook3, SteinTreePacking}}, Hamiltonicity~\cite{diestel2015book,AgelosFleischner,GeoTopCircles,BruhnYu,HeuerTwoExtensions,ErdeLehnerPitz,CuiWangYuHamilton}, and various planarity criteria~\cite{duality,BruhnSteinMacLane,DiestelPottDualgraphs}.

A similarly useful notion for ends of digraphs has been found only very recently. 
In a series of three papers~\cite{EndsOfDigraphsI,EndsOfDigraphsII,EndsOfDigraphsIII}, B\"urger and the author introduced a notion of ends in digraphs for which the fundamental techniques of undirected end space theory naturally generalise to  digraphs. Unlike for undirected graphs, some ends of digraphs are joined by \emph{limit  edges}. A digraph $D$ together with its ends and limit edges naturally forms a topological space $|D|$. So the scene is set now to attempt, also for digraphs $D$, to extend finite to infinite theorems by letting the  naturally oriented topological paths and circles in $|D|$ take the role of directed paths and  cycles in $D$. The purpose of this paper is to make a start on this programme, with two well-known Hamiltonicity theorems for digraphs.

Two folklore theorems in finite graph theory, due to R\'edei~\cite{Redei} and Camion~\cite{Camion}, respectively, say that every finite tournament has a Hamilton path, and every finite strongly connected tournament has a Hamilton cycle. In this paper we show that these results have natural analogues in the space $|D|$. We shall see that ends and limit edges are both crucial for such extensions to exist: there exists a countable tournament $D$ whose compactification by just the ends of the underlying undirected graph contains no topological Hamilton path. (Similarly, $D$ has no topological Hamilton path in $|D|$ that avoids all its limit edges, and $D$ has no spanning ray or double ray.)

To state our results formally, we need a few definitions. A \emph{ray} is an infinite directed path that has a first vertex (but no last vertex). The  subrays of a  ray are its \emph{tails}. A ray in a digraph $D$ is \emph{solid} in $D$ if it has a tail in some strong component of $D-X$ for every finite vertex set $X\subseteq V(D)$. Two solid rays in  $D$ are \emph{equivalent} if for every finite vertex set $X\subseteq V(D)$ they have a tail in the same strong component of $D-X$. The classes of this equivalence relation are the \emph{ends} of $D$. For an end $\omega$ we write $C(X,\omega)$ for the strong component of $D-X$ in which every ray that represents $\omega$ has a tail. For two ends $\omega$ and $\eta$ of $D$ a finite vertex set $X \subseteq V(D)$ is said to \emph{separate} $\omega$ and $\eta$ if $C(X, \omega) \neq C(X, \eta )$. For two distinct ends $\omega$ and $\eta$ of $D$ we call the pair  $(\omega,\eta)$ a \emph{limit edge} of $D$ from $\omega$ to $\eta$ if $D$ has an edge from $C(X,\omega)$ to $C(X,\eta)$ for every finite vertex set $X \subseteq V(D)$ that separates $\omega$ and $\eta$. For a vertex $v\in V(D)$ and an end $\omega$ we call the pair $(v,\omega)$ a \emph{limit edge} of $D$ from $v$ to $\omega$ if $D$ has an edge from $v$ to $C(X,\omega)$ for every finite vertex set $X\subseteq V(D)$ with $v \not\in  C(X, \omega)$. Similarly, we call the pair $(\omega,v)$ a \emph{limit edge} of $D$ from $\omega$ to $v$ if $D$ has an edge from $C(X,\omega)$ to $v$ for every finite vertex set $X\subseteq V(D)$ with $v \not\in C(X, \omega)$. For example if $R$ is a ray and every vertex of $R$ sends an edge to a vertex $v$, then there is a limit $(\omega,v)$ from the end $\omega$ that is represented by $R$ to $v$.

The topological space $|D|$ has as its ground set the digraph $D$, viewed as a  \hbox{1-complex}, together with the ends and limit edges of $D$. The topology on $|D|$ will be defined formally in Section~\ref{section: Preliminaries}.

A \emph{topological path in $|D|$} is a continuous map $\alpha \colon [0,1] \to |D|$ that respects the direction of the edges of $D$ when it traverses them. For example, a ray that represents an end $\omega$ naturally defines a topological path from its first vertex to $\omega$, and might be extended  by a limit edge that starts at $\omega$. A  \emph{Hamilton path in}  $|D|$ is an injective topological path  in $|D|$ that traverses every vertex exactly once. We remark that, as every end of $D$ is a limit point of  vertices of $D$, any topological path that traverses all the vertices of $D$ also traverses all its ends.

There are two trivial obstacles for $|D|$ to containing a  Hamilton path. The first is that the  cardinality of $D$ may be larger than the cardinality of the unit interval. For this reason we will only consider countable digraphs: these can have continuum many ends, but no more. 
Another potential obstruction to the existence of a Hamilton path in $|D|$ is that the space $|D|$ may not be compact. As any continuous image of $[0,1]$ is compact, it is not hard to show that $|D|$ is compact as soon as a topological path traverses all the vertices of $D$. (We shall disallow parallel edges, as these do not affect hamiltonicity.) For this reason we will only consider those digraphs for which $|D|$ is a compactification of $D$. These can be described combinatorially, as fallows.

A digraph $D$ is called \emph{solid} if $D-X$ has only finitely many strong components for all finite vertex sets $X \subseteq V(D)$. As shown in~\cite[Theorem~1]{EndsOfDigraphsII}, a digraph $D$ is solid if and only if $|D|$ is compact.  A vertex $v$ \emph{can reach} a vertex $w$ in $D$ if there is a (finite) path in $D$ from $v$ to $w$. Our first main theorem reads as follows:

\begin{customthm}{1}\label{introTheoremOne}
Every countable solid tournament has a  Hamilton path. This  Hamilton path may be chosen so as to start at any vertex that can reach every other vertex.
\end{customthm}

For finite tournaments $D$ there is a standard proof of this result: given a directed path $P$ in $D$, a quick case distinction shows that any vertex not yet contained in $P$ can be inserted into $P$. 
This proof strategy does not adapt easily to infinite tournaments $D$. It is still possible to insert vertices one after the other to obtain a sequence of longer and longer finite paths which, eventually, contain all the vertices of $D$. But even if we can show that these paths converge to a topological path in $|D|$ that contains all its vertices, this need not be a  Hamilton path by our definition: it might visit some ends multiple times, and thus fail to be injective.

But there is another proof for the finite case, which, as we shall see, can be adapted to infinite digraphs. Every finite tournament $D$ contains a vertex $r$ that can reach every other vertex of $D$. Let $T$ be a tree obtained by a depth-first search starting at $r$; note that $r$ can reach every vertex of $D$ even in $T$. Now $T$ imposes a partial ordering $\leq_T$ on $V:=V(D)=V(T)$ defined by letting $v \leq_T w$ if $v$ lies on the path in $T$ from $r$ to $w$. As $T$ is a depth-first search tree, this order $\leq_T$ has a linear extension $\leq$ on $V$ in which $u \geq v$ for $\leq_T$-incomparable vertices $u,v$ if our search found $u$ before $v$. (Note that this differs from when $u$ and $v$ are $\leq_T$-comparable: in that case we have $u \leq_T v$, and hence $u \leq v$, if our search found $u$ before $v$.)

This is indeed a total order; it is known as the \emph{reverse post-order} and widely used in computer science. Crucially for us one can show that, if $v$ is the predecessor of $w$ in $\leq$, the unique edge of $D$ between them is directed from $v$ to $w$~\cite{cormen2009book}. Clearly, therefore, this total order on $V$ defines a directed Hamilton path in $D$. Let us see now how the above proof adapts to infinite digraphs in our topological setting.

In the third paper of the series~\cite{EndsOfDigraphsIII}, depth-first search trees were adapted to infinite digraphs; these infinite analogues are called \emph{normal} arborescences. (An \emph{arborescence}, in any digraph, is an oriented rooted tree in which the root can reach every vertex.) The notion of normality will be introduced in Section~\ref{section: Preliminaries}. For now we only need that normal spanning arborescences exist in every countable solid tournament; they define a  tree-order on the vertices as in finite digraphs, and this ordering extends to a total order on its vertices and ends in which $(x,y)$ is an (oriented) edge or limit edge whenever $x$ is the predecessor of $y$. To prove Theorem~\ref{introTheoremOne} it then only remains to show that this ordering is continuous at ends, and thus defines a Hamilton path in $|D|$ as before.

To illustrate Theorem~\ref{introTheoremOne}, let us look at an example. Let $D$ be a  solid tournament in which the infinite binary tree $T$ is a normal spanning arborescence. Then $D$ has a Hamilton path $\alpha$ which traverses every tree-edge from a vertex $v$ to its right child $v|1$, and every limit edge from an end $\omega_{v}$ represented by a ray $v|1000\ldots$ to the vertex $v|0$; see Figure~\ref{fig: backtracking}. This $\alpha$ can be viewed as a limit of the Hamilton paths discussed earlier of the finite subtournaments $D_n$ of $D$ spanned by the  subtrees $T_n$ of height $n$ in $T$, which are depth-first search trees of $D_n$.\begin{figure}[ht]
    \centering
    \begin{tikzpicture}[scale=0.72]
    \begin{scope}[xscale=-1,yscale=1]
    \tikzset{edge/.style = {->,> = latex'}}
    
     \draw[fill,black] (0,-2.5) circle (.05);
     
     \draw[edge, thick] (0,-2.5) to (-2.8,-1.5);
     \draw[edge, thick] (0,-2.5) to (2.8,-1.5);

     \draw[fill,black] (-2.8,-1.5) circle (.05);
     \node at (-3.1,-1.6) {$v$};
     \draw[fill,black] (2.8,-1.5) circle (.05);
     \node (R) at (2.8,-1.5) {};
     
     \draw[edge, thick] (-2.8,-1.5) to (-4.1,0);
     \draw[edge, thick] (-2.8,-1.5) to (-1.5,0);
     \node at (-1.1,-0.2) {$v|0$};
     \node at (-4.5,-0.2) {$v|1$};
     
     \node at (-3.6,2.25) {$\omega_v$};
     
     \draw[edge, thick] (2.8,-1.5) to (4.1,0);
     \draw[edge, thick] (2.8,-1.5) to (1.5,0);

\begin{scope}[shift={(-4.1,0)}]
     \draw[fill,black] (0,0) circle (.05);
     
      \draw[edge, thick,rotate around={25:(0,0)}] (0,0) to (0,1.5);
      \draw[edge, thick,rotate around={-25:(0,0)}] (0,0) to (0,1.5);
      
      \node at (115:1.6) {$\cdot$};
      \node at (115:1.75) {$\cdot$};
      \node at (115:1.9) {$\cdot$};
      
      \node at (65:1.6) {$\cdot$};
      \node at (65:1.75) {$\cdot$};
      \node at (65:1.9) {$\cdot$};
      
      
      \draw[fill,black] (115:2.15) circle (.05);
      \node (LLL) at  (115:2.15) {};
      \draw[fill,black] (65:2.15) circle (.05);
      \node (LLR) at  (65:2.15) {};
      
\end{scope}

\begin{scope}[shift={(-1.5,0)}]
     \draw[fill,black] (0,0) circle (.05);
     \node (LR) at (0,0) {};
     
      \draw[edge, thick,rotate around={25:(0,0)}] (0,0) to (0,1.5);
      \draw[edge, thick,rotate around={-25:(0,0)}] (0,0) to (0,1.5);
      
      \node at (115:1.6) {$\cdot$};
      \node at (115:1.75) {$\cdot$};
      \node at (115:1.9) {$\cdot$};
      
      \node at (65:1.6) {$\cdot$};
      \node at (65:1.75) {$\cdot$};
      \node at (65:1.9) {$\cdot$};
      
      
      \draw[fill,black] (115:2.15) circle (.05);
      \draw[fill,black] (65:2.15) circle (.05);
      \node (LRR) at (65:2.15) {};
\end{scope}

\begin{scope}[shift={(1.5,0)}]
     \draw[fill,black] (0,0) circle (.05);
     
      \draw[edge, thick,rotate around={25:(0,0)}] (0,0) to (0,1.5);
      \draw[edge, thick,rotate around={-25:(0,0)}] (0,0) to (0,1.5);
      
      \node at (115:1.6) {$\cdot$};
      \node at (115:1.75) {$\cdot$};
      \node at (115:1.9) {$\cdot$};
      
      \node at (65:1.6) {$\cdot$};
      \node at (65:1.75) {$\cdot$};
      \node at (65:1.9) {$\cdot$};
      
      
      \draw[fill,black] (115:2.15) circle (.05);
      \draw[fill,black] (65:2.15) circle (.05);
      \node (RLR) at (65:2.15) {};
\end{scope}    

\begin{scope}[shift={(4.1,0)}]
     \draw[fill,black] (0,0) circle (.05);
     \node (RR) at (0,0) {};
     
      \draw[edge, thick,rotate around={25:(0,0)}] (0,0) to (0,1.5);
      \draw[edge, thick,rotate around={-25:(0,0)}] (0,0) to (0,1.5);
      
      \node at (115:1.6) {$\cdot$};
      \node at (115:1.75) {$\cdot$};
      \node at (115:1.9) {$\cdot$};
      
      \node at (65:1.6) {$\cdot$};
      \node at (65:1.75) {$\cdot$};
      \node at (65:1.9) {$\cdot$};
      
      
      \draw[fill,black] (115:2.15) circle (.05);
       \draw[fill,black] (65:2.15) circle (.05);
\end{scope}

    \draw[edge, thick, dashed] (-3.2,1.95) to[bend right] (-1.5,0);
    \draw[edge, thick, dashed] (-0.6,1.95) to[bend right] (2.8,-1.5);
    \draw[edge, thick, dashed] (2.4,1.95) to[bend right] (4.1,0);
    
   \draw[line width=2.5, gray,opacity=0.6] (-3.2,1.95) to[bend right] (-1.5,0);
    \draw[line width=2.5, gray,opacity=0.6] (-0.6,1.95) to[bend right] (2.8,-1.5);
    \draw[line width=2.5, gray,opacity=0.6] (2.4,1.95) to[bend right] (4.1,0);
    
    \draw[edge, line width=3, gray,opacity=0.6] (0,-2.5) to (-2.8,-1.5) to (-4.1,0 ) to (-4.75,1.4);
    
     \draw[edge, line width=3, gray,opacity=0.6]  (-1.5,0 ) to (-2.15,1.4);
    
     \draw[edge, line width=3, gray,opacity=0.6]  (2.8,-1.5) to (1.5,0 ) to (0.85,1.4);
 
     \draw[edge, line width=3, gray,opacity=0.6]  (4.1,0 ) to (3.45,1.4);


\node (A) at  (-5,1.95) {};

\node (B) at  (-4,0.8) {};

\node (C) at  (-4.4,1.95) {};

\node (D) at  (-3.75,1.325) {};

\node (F) at  (-3.85,1.95) {};

\node (G) at  (-3.55,1.75)  {};

\node (H) at  (-3.45,1.95)  {};

\node (I) at  (-3.2,1.95)  {};




\draw[line width=2,  gray, opacity=0.7] plot [smooth, tension=.9] coordinates{(A)  (B)  (C) (D) (F) (G) (H) (I)};

\begin{scope}[shift={(2.6,0)}]
\node (A) at  (-5,1.95) {};

\node (B) at  (-4,0.8) {};

\node (C) at  (-4.4,1.95) {};

\node (D) at  (-3.75,1.325) {};

\node (F) at  (-3.85,1.95) {};

\node (G) at  (-3.55,1.75)  {};

\node (H) at  (-3.45,1.95)  {};

\node (I) at  (-3.2,1.95)  {};




\draw[line width=2,  gray, opacity=0.7] plot [smooth, tension=.9] coordinates{(A)  (B)  (C) (D) (F) (G) (H) (I)};

\end{scope}

\begin{scope}[shift={(5.6,0)}]
\node (A) at  (-5,1.95) {};

\node (B) at  (-4,0.8) {};

\node (C) at  (-4.4,1.95) {};

\node (D) at  (-3.75,1.325) {};

\node (F) at  (-3.85,1.95) {};

\node (G) at  (-3.55,1.75)  {};

\node (H) at  (-3.45,1.95)  {};

\node (I) at  (-3.2,1.95)  {};




\draw[line width=2,  gray, opacity=0.7] plot [smooth, tension=.9] coordinates{(A)  (B)  (C) (D) (F) (G) (H) (I)};

\end{scope}

\begin{scope}[shift={(8.2,0)}]
\node (A) at  (-5,1.95) {};

\node (B) at  (-4,0.8) {};

\node (C) at  (-4.4,1.95) {};

\node (D) at  (-3.75,1.325) {};

\node (F) at  (-3.85,1.95) {};

\node (G) at  (-3.55,1.75)  {};

\node (H) at  (-3.45,1.95)  {};

\node (I) at  (-3.2,1.95)  {};




\draw[line width=2,  gray, opacity=0.7] plot [smooth, tension=.9] coordinates{(A)  (B)  (C) (D) (F) (G) (H) (I)};

\end{scope}
    \end{scope}
   \end{tikzpicture}
\caption{A tournament with the infinite binary tree as a normal spanning arborescence. All edges between $\leq_T$-incomparable vertices  run from right to left. 
}\label{fig: backtracking}
\end{figure}
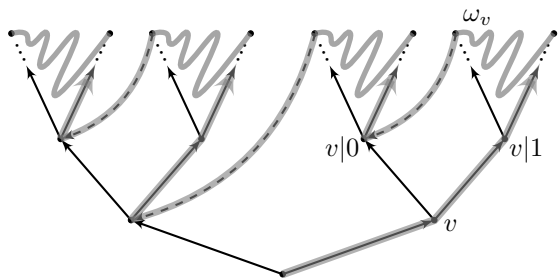

A topological path $\alpha$ in $|D|$  is \emph{closed} if $\alpha(0)=\alpha(1)$.  A \emph{Hamilton circle} of $D$ is a closed but otherwise injective topological path in $|D|$ that traverses every vertex. As remarked earlier, this implies that it also traverses every end (exactly once).

\begin{customthm}{2}\label{introTheoremTwo}
Every countable strongly connected solid tournament has a  Hamilton  circle. 
\end{customthm}

As shown in~\cite[Theorem~4,~Lemma~5.1]{EndsOfDigraphsII},  a countable solid digraph $D$ is strongly connected if and only if for any two points $x,y \in |D|$ there is a topological path in $|D|$ from $x$ to $y$.  Again, there is a standard proof of Theorem~\ref{mainthm: Camion} for finite tournaments $D$: a quick case distinction shows that any vertex not yet contained in a given cycle can be inserted. Again, for infinite $D$ it is possible to insert new vertices, one after the other, into a cycle to obtain at the limit a  closed topological path containing all the vertices. But, as earlier, this topological path might traverse ends multiple times.

We will instead use Theorem~\ref{introTheoremOne} to prove  Theorem~\ref{introTheoremTwo}. Ideally, we would like to fix a  Hamilton path $\alpha$ in $|D|$ and then use an edge or a limit edge from the endpoint of $\alpha$ to its starting point  to obtain a Hamilton circle in $|D|$. This is not always possible, as we are only free to choose the starting point of $\alpha$. However, an analysis of $D$ and its ends will give us enough control over the endpoint of $\alpha$ to construct the desired Hamilton circle.

This paper can be read without reading the  series~\cite{EndsOfDigraphsI,EndsOfDigraphsII,EndsOfDigraphsIII} about ends of digraphs first. We only need a few terms and results from this series, which we collect in Section~\ref{section: Preliminaries}. We will then prove Theorem~\ref{mainthm: Redei} in Section~\ref{section: Rédei's theorem}, and Theorem~\ref{mainthm: Camion} in Section~\ref{section: Camion's theorem}.

\section{Preliminaries}\label{section: Preliminaries}

\noindent For graph-theoretic terms we follow the terminology in \cite{diestel2015book}. Throughout this paper, $D$ is an infinite digraph without infinitely many parallel edges and without loops. We write $V(D)$ for its vertex set,  $E(D)$ for its edge set, $\Omega(D)$ for its set of ends and $\Lambda(D)$ for its set of limit edges. For a finite vertex set $X$ and an end $\omega$ of $D$ we write $C(X, \omega)$ for the strong component of $D-X$ that contains a tail of every ray that represents $\omega$. We write $\Omega(X,\omega)$ for the set of  ends which are represented by a ray in $C(X,\omega)$.

In the following, we give a concise definition of the space $|D|$ and its topology. For a detailed introduction of the space $|D|$ and its topology, see~\hbox{\cite[Section~2~and~3]{EndsOfDigraphsII}}. The ground set of $|D|$ is obtained by taking $V(D) \cup \Omega(D)$ together with a copy of the unit interval $[0,1]_e$ for every edge or limit edge $e$ of $D$. Then we identify every vertex or end $x$ with the copy of 0 in $[0,1]_e$ for which $x$ is the tail of $e$ and with the copy of 1 in $[0,1]_f$ for which $x$ is the head of $f$, for all $e,f \in E(D) \cup \Lambda(D)$. Basic open sets of vertices $v$ are uniform  stars of radius $\varepsilon$ around $v$, i.e. an  $\varepsilon$ length from every edge or limit edge that is adjacent to $v$. Basic open sets of inner points of edges $e$ are open subintervals of $[1,0]_e$ containing it. For ends $\omega$, basic open sets $\hat{C}_\varepsilon(X,\omega)$ are the union of $C(X,\omega)$ together with $\Omega(X,\omega)$ and every limit edge which has both its endpoints in $C(X,\omega) \cup \Omega(X,\omega)$ and an $\varepsilon$ length of every edge or limit edge which has precisely one endpoint in $C(X, \omega) \cup \Omega(X,\omega)$, for finite vertex sets $X \subseteq V(D)$. For inner points $z$ of limit edges $(\omega, \eta)$, basic open sets $\hat{E}_{\varepsilon ,z} (X,(\omega,\eta))$ are the union of $\varepsilon$ intervals around the copy of $z$ in every edge between $C(X,\omega) \cup \Omega (X,\omega)$ and $C(X,\eta) \cup \Omega (X,\eta)$, for finite vertex sets $X$ which separate $\omega$ and $\eta$. Similarly, for inner points $z$ of limit edges $(v,\omega)$, basic open sets $\hat{E}_{\varepsilon ,z} (X,(v,\omega))$ are the union of $\varepsilon$ intervals around the copy of $z$ in every edge or limit edge between $v$ and $C(X,\omega) \cup \Omega(X,\omega)$, for finite vertex sets $X$ which contain~$v$. Basic open sets $\hat{E}_{\varepsilon ,z} (X,(\omega,v))$ are defined analogously.

\begin{figure}[ht]
    \centering
    \begin{tikzpicture}[scale=0.8]
    \tikzset{edge/.style = {->,> = latex'}}


\draw (-0.8 , 0.5) -- (0.8 ,0.5) -- (0.8 ,-0.3) -- (-0.8 ,-0.3) -- (-0.8,0.5);

\begin{scope}[shift={(-2.1,1.2)}]
\draw (-1.3,3) to[out=-90, in=180] (0,0);
\draw (1.3,3) to[out=-90, in=0] (0,0);
\end{scope}

\begin{scope}[shift={(2.1,1.2)}]
\draw (-1.3,3) to[out=-90, in=180] (0,0);
\draw (1.3,3) to[out=-90, in=0] (0,0);
\end{scope}


 \draw[fill,black] (-1.3,4.2) circle (.05);
\node at (-1.5,4.4) {$\omega$};

 \draw[fill,black] (1.3,4.2) circle (.05);
\node at (1.5,4.4) {$\eta$};

\draw[edge,thick,dashed] (-1.3,4.2) to[bend left] (1.3,4.2);

    \draw[fill,black] (-1.3,2.2) circle (.05);
    \draw[fill,black] (-1.3,2.75) circle (.05);
    \draw[fill,black] (-1.3,3.2) circle (.05);
    
    \draw[fill,black] (1.3,2.2) circle (.05);
    \draw[fill,black] (1.3,2.75) circle (.05);
    \draw[fill,black] (1.3,3.2) circle (.05);
    
    \draw[edge,thick] (-1.3,2.2) to[bend left] (1.3,2.2);
    \draw[edge,thick] (-1.3,2.75) to[bend left] (1.3,2.75);
    \draw[edge,thick] (-1.3,3.2) to[bend left] (1.3,3.2);
    
    \node at (0,4.23) {$\vdots$};

 \draw[fill,black] (3.1,4.2) circle (.05);
\node at (2.85,4.4) {$\omega'$};

\draw[fill,black] (0.6,0.3) circle (.05);

 \draw[fill,black] (2,2.2) circle (.05);
    \draw[fill,black] (2.4,2.75) circle (.05);
    \draw[fill,black] (2.8,3.2) circle (.05);
    
\draw[edge,thick]    (0.6,0.3) to[out=30, in=-105 ]  (2,2.2);

\draw [edge,thick] plot [smooth, tension=.8] coordinates {  (0.6,0.3) (1.85,1)  (2.4,2.75)};


\draw [edge,thick] plot [smooth, tension=.8] coordinates {  (0.6,0.3) (2.25,1) (2.8,3.2)};

\draw [edge,thick,dashed] plot [smooth, tension=.8] coordinates {  (0.6,0.3)  (2.6,.9)  (3.1,4.2)};


\draw[fill,black] (-0.6,0.3) circle (.05);
\draw[edge,thick] (-0.6,0.3) to[bend right] (0.6,0.3);


\draw[fill,black] (-1.8,1.8) circle (.05);
\draw[fill,black] (-2.4,1.8) circle (.05);

\draw[edge,thick] (-0.6,0.3) to[bend left] (-1.8,1.8);
\draw[edge,thick] (-0.6,0.3) to[bend left] (-2.4,1.8);


\draw[fill,black] (-2.2,4.2) circle (.05);
\draw[fill,black] (-3.1,4.2) circle (.05);

\draw[edge,thick,dashed] (-3.1,4.2) to[bend left] (-2.2,4.2);


 \draw[fill,black] (-2.9,2.2) circle (.05);
  \draw[fill,black] (-2.9,2.75) circle (.05);
  \draw[fill,black] (-2.9,3.2) circle (.05);

\draw[edge,thick] (-4,2.2) to[bend left] (-2.9,2.2);
\draw[edge,thick] (-4,2.75) to[bend left] (-2.9,2.75);
\draw[edge,thick] (-4,3.2) to[bend left] (-2.9,3.2);

\draw[edge,thick,dashed] (-4,4.2) to[bend left] (-3.1,4.2);



\begin{scope}[shift={(-2.1,1.2)}]
\fill[fill=gray, fill opacity=0.2] (-1.5,3.5) to[out=-90, in=180] (0,-.2) to[out=0, in=-90](1.5,3.5) to (-1.5,3.5);
\end{scope}

\draw[double distance = 5,line cap=round,color=gray, opacity=0.3] (-.2,4.58) to[out=13, in=167] (.2,4.58);

\draw[double distance = 5,line cap=round,color=gray, opacity=0.3] (-.2,3.58) to[out=13, in=167] (.2,3.58);

\draw[double distance = 5,line cap=round,color=gray, opacity=0.3] (-.2,3.12) to[out=13, in=167] (.2,3.12);

\draw[double distance = 5,line cap=round,color=gray, opacity=0.3] (-.2,2.58) to[out=13, in=167] (.2,2.58);


\node at (-1.1,-.3) {$X$};

\node at (0.6,0) {\footnotesize{$v$}};

\node at (-4,5.2) {\footnotesize{$\hat{C}_\varepsilon(X,\omega) $ }}  ;

\node at (.5,5.2) {\footnotesize{$\hat{E}_{\varepsilon,z}(X,(\omega,\eta)) $ }}  ;


    \begin{scope}[rotate around={-130:(2.5,.8)} , shift={(2.5,.8)}]
    \draw[double distance = 5,line cap=round,color=gray, opacity=0.3] (-.2,0) to[out=13, in=167] (.2,0);
    \end{scope}
    
    \begin{scope}[rotate around={-138:(2.05,.83)} , shift={(2.05,.83)}]
    \draw[double distance = 5,line cap=round,color=gray, opacity=0.3] (-.2,0) to[out=13, in=167] (.2,0);
    \end{scope}
    
    \begin{scope}[rotate around={-140:(1.65,.83)} , shift={(1.65,.83)}]
    \draw[double distance = 5,line cap=round,color=gray, opacity=0.3] (-.2,0) to[out=13, in=167] (.2,0);
    \end{scope}
    
    \begin{scope}[rotate around={-138:(1.3,.85)} , shift={(1.3,.85)}]
    \draw[double distance = 5,line cap=round,color=gray, opacity=0.3] (-.2,0) to[out=13, in=167] (.2,0);
    \end{scope}
    
    \node at (4.3,1) {\footnotesize{$\hat{E}_{\varepsilon,z}(X,(v,\omega')) $} }  ;

    \end{tikzpicture}
\caption{Basic open sets for ends and inner points of limit edges.}\label{fig:opensets}
\end{figure}
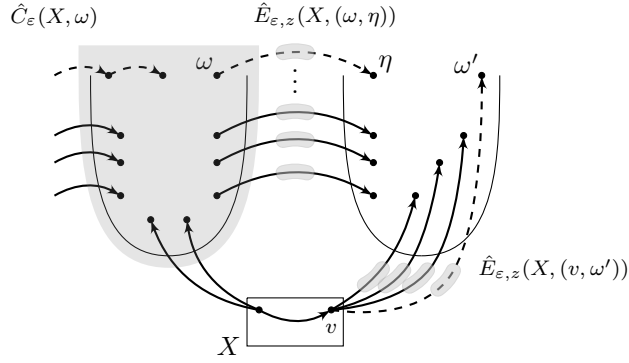

We will find the desired Hamilton path in Theorem~\ref{introTheoremOne} via an inverse limit construction. As shown in~\cite[Section~4]{EndsOfDigraphsII} for solid $D$, the space $|D|$ is the inverse limit of its finite contraction minors. In the following, we will give a concise recap of this result for countable digraphs. For the general definition of an inverse system and its inverse limit, see~\cite[Chapter~8]{diestel2015book} and for their topological properties, see~\cite{RibesZalesskii}.
Let $D$ be a countable digraph, fix any enumeration of its vertex set and write $X_n$ for the set of the first $n$ vertices. We denote by $P_n$ the  partition of $V(D)$ where each vertex in $X_n$ is a singleton partition class and the other partition classes consist of the strong components of $D-X_n$. Every such partition $P_n$ gives rise to a finite (multi-)digraph $D/P_n$ by contracting each partition class and replacing the edges running from a partition class to another by a single edge whenever there are infinitely many.  Formally, declare $P_n$ to be the vertex set of $D/P_n$.  Given distinct partition classes $p_1,p_2 \in P_n$, we define an edge $(e,p_1,p_2)$ of $D/P_n$ for every edge $e \in  E(D)$ from $p_1 $ to $p_2$ if there are  finitely  many  such  edges. If  there  are  infinitely  many  edges  from $p_1$ to $p_2$ we  just  define  a  single  edge  $(p_1p_2,p_1,p_2)$. We  call  the  latter  type  of  edges \emph{quotient edges}. Endowing $D/P_n$ with the 1-complex topology turns it into a compact Hausdorff space, i.e. basic open sets are uniform $\varepsilon$ stars around vertices and open subintervals  of edges. For $m \leq n$  there is a  map  $f_{n, m}$ from $V(D /P_n)$ to  $V( D /P_{m})$, mapping every vertex of $D /P_n$ to the vertex of $D /P_{m}$ containing it. This map extends naturally to a continuous map from $D / P_n$ to $D / P_{m}$ by mapping edges of $D/ P_n$ to vertices or edges of $D/ P_m$ according to the images of its endpoints. This gives an inverse system and if $D$ is solid, its inverse limit coincides with $|D|$:

\begin{corollary}[{\cite[Corollary~4.4]{EndsOfDigraphsII}}]\label{cor: modD inverse limit of deletion minors}
Let $D$ be a countable solid digraph and let  $X_n$ consist of the first $n$ vertices of $D$ with regard to any fixed  enumeration of $V(D)$. Then $|D|~\cong~\varprojlim(D/ P_{n})_{n \in \N } $.
\end{corollary}

Finally, let us recap the notion of normal arborescences from~\cite{EndsOfDigraphsIII}. An \emph{arborescence} is a rooted oriented tree $T$ that contains for every vertex \mbox{$v \in V(T)$} a directed path from the root to $v$.  The vertices of any arborescence are partially ordered as $v \leq_T w$ if $T$ contains a directed path from $v$ to $w$. We write $\uc{v}_T$ for the up-closure and $\dc{v}_T$ for the down-closure of $v$ in $T$.  Consider a  digraph $D$ and a spanning arborescence $T\subseteq D$. The \emph{normal assistant} of $T$ in $D$ is the auxiliary digraph $H$ that is obtained from $T$ by adding an edge $(v,w)$ for every two $\le_T$-incomparable vertices $v,w \in V(T)$ for which there is an edge from $\uc{v}_T$ to $\uc{w}_T$ in $D$, regardless of whether $D$ contains such an edge. The spanning arborescence $T$ is \emph{normal} in $D$ if the normal assistant of $T$ in $D$ is acyclic; in this case, the transitive closure of the normal assistant defines a partial order $\nao$ on the vertices of $D$ and we call $\nao$ the \emph{normal order} of $T$. We remark that if $D$ is a finite spanning arborescence it is normal in $D$ if and only if it defines a depth-first search tree, see~\cite[Corollary~3.3]{EndsOfDigraphsIII}. One of the most useful properties of normal arborescences is that they capture the separation properties of their host graph~\cite[Lemma~3.4]{EndsOfDigraphsIII}:

\begin{lemma}\label{lemma: separation props of NSA}
Let $D$ be any digraph and let $T\subseteq D$ be a normal spanning arborescence in $D$. If $v,w\in V(T)$ are $\le_T$-incomparable vertices of $T$ with $w\ntrianglelefteq_T  v$, then every path from $w$ to $v$ in $D$ meets $X:=\dc{v}_T\cap \dc{w}_T$. In particular, $X$ separates $v$ and $w$ in $D$. 
\end{lemma}

Not every digraph has a normal spanning arborescence. However, as a direct consequence of \cite[Theorem~3]{EndsOfDigraphsIII}, we have that all countable digraphs have one:


\begin{lemma}\label{lem: countabledigraphs have NSA}
Let $D$ be a countable digraph and $r \in V(D)$ a vertex that can reach every other vertex in $D$. Then $D$ has a normal spanning arborescence rooted in $r$.
\end{lemma}

If $T$ is a normal spanning arborescence of a solid digraph $D$, then every ray of $T$ is solid in $D$ and therefore represents an end of $D$. By Lemma~\ref{lemma: separation props of NSA}, any two distinct rays in $T$ that start at the root represent distinct ends of $D$. Conversely, it is shown in~\cite[Theorem~1]{EndsOfDigraphsIII} that any end of $D$ is represented by a ray in $T$:

\begin{lemma}\label{lem: NST endfaith}
Let $D$ be any digraph and $T$ a normal spanning arborescence of $D$. Then for every end of $D$ there is exactly one ray in $T$ that represents the end in $D$ and starts at the root of $T$.
\end{lemma}

\section{Hamilton paths}\label{section: Rédei's theorem}

\noindent In this section we prove Theorem~\ref{introTheoremOne} and give an example which shows that ends and limit edges are crucial for such an extension to exist. As mentioned in the introduction we will find the desired Hamilton path of Theorem~\ref{introTheoremOne} alongside a normal spanning arborescence of the tournament. This makes it possible to prove a slightly stronger statement and we will need this strengthening in our proof of Theorem~\ref{mainthm: Camion}. For a tournament $D$ with a normal spanning arborescence $T$, we say that an injective topological path $\alpha$ in $|D|$ \emph{respects the normal order} of $T$ if $\alpha$ traverses a vertex $t$ before a vertex $t'$ if and only if $t$ is less than $t'$ in the normal order of $T$. Note that the normal order of $T$ is a total order, as $D$ is a tournament.

\begin{mainresult}\label{mainthm: Redei}
Let $D$ be a countable solid tournament with a normal spanning arborescence $T$. Then $D$ has a Hamilton path that respects the normal order of~$T$.
\end{mainresult}

Note that every solid tournament $D$ has a vertex that can reach every other vertex. Indeed, as $D$ is solid it has only finitely many strong components and one of them sends an edge to every other strong component; every vertex in this component can reach any other vertex in $D$. Since $D$ is countable, it has a normal spanning arborescence $T$ rooted at any given vertex that can reach every other vertex of $D$, Lemma~\ref{lem: countabledigraphs have NSA}. And any Hamilton path that respects the normal order of $T$ starts at the root of $T$, since it is the smallest element. So the above formulation of Theorem~\ref{mainthm: Redei} implies the formulation in the introduction.

\begin{proof}

Our goal is to show that the normal order of $T$ naturally defines a Hamilton path in $|D|$. We will show this by an inverse limit construction. It is straightforward to find a sequence $X_1 \subseteq X_2 \subseteq \ldots $ of finite vertex sets of $V(D)$ such that: 
\begin{enumerate}
    \item the union of all the $X_n$ is $V(D)$ and
    
    \item $X_n$ is down-closed in $T$ with regards to the tree-order.
\end{enumerate}

Now, every $X_n$ defines a partition $P_n$ of $V(D)$ and a finite contraction minor $D/P_n$ of $D$ as in Section~\ref{section: Preliminaries}. These $D/ P_n$ form an inverse system with bonding maps $f_{n,m}$. By the first property (i), the partitions $P_n$ are cofinal in a sequence of partitions that arise by an enumeration of $V(D)$. Hence, the $D/P_n$ form a cofinal (sub-)inverse system of an inverse system that arises by an enumeration of $V(D)$; so the inverse limit of both inverse systems coincides and we have by Corollary~\ref{cor: modD inverse limit of deletion minors} that $|D|~\cong~\varprojlim(D/ P_{n})_{n \in \N }$. Next, we will find compatible Hamilton paths in every $D/P_n$ so that the universal property of the inverse limit gives the desired Hamilton path in $|D|$. By the second property (ii) and Lemma~\ref{lemma: separation props of NSA}, the edges of $T$ in $D/P_n$ form a spanning arborescence $T_n$ of $D/P_n$. Moreover, as $T$ is normal in $D$ we have that $T_n$ is normal in $D/P_n$. As $D$ is a tournament we have that the normal order of $T_n$ is a total order on the vertices of $D/P_n$. Let $v_1 \trianglelefteq_{T_n}  \ldots \trianglelefteq_{T_n}  v_k$ be the sequence of  vertices of  $D / P_n$ ordered by the normal order of $T_n$.\footnote{We remark that this is the reverse post-order of $T_n$.} We claim that $W_n = v_1, \ldots ,v_k$ is a Hamilton path in $D / P_n$. Indeed, either $v_i$ and $v_{i+1}$ are $\leq_{T_n}$-comparable in which case $(v_i,v_{i+1})$ is a tree-edge of $T_n$ or  $v_i$ and $v_{i+1}$ are $\leq_{T_n}$-incomparable, in which case there is a cross-edge from $v_i$ to $v_{i+1}$ as $D$ is a tournament. These Hamilton paths are compatible in the sense that $f_{n,m}(v_1 )\trianglelefteq_{T_m} \ldots \trianglelefteq_{T_m} f_{n,m}(v_k)$ is the sequence of vertices of $W_m$. However, as there might be parallel cross-edges from $v_i$ to $v_{i+1}$, it might happen that $(f_{n,m}(v_i),f_{n,m}(v_{i+1})$ does not coincide with the edge from $f_{n,m}(v_i)$ to $f_{n,m}(v_{i+1})$ in $W_m$. However, for every $D/P_n$ there are only finitely many Hamilton paths with vertex sequence $v_1,\ldots,v_k$. So by  K\H{o}nig's infinity lemma, we might choose the edges of $W_n$ such that $f_{n,m}(W_n)$ gives $W_m$.

Finally, fix for every $n\in \N$ a parameterisation $\alpha_n \colon [0,1] \to D / P_n$ of $W_n$. It is straightforward to choose the $\alpha_n$ in a compatible way, i.e. the projection of $\alpha_n$ coincides with $\alpha_{n-1}$. Moreover, we may choose the $\alpha_n$ so that they are nowhere constant but on the strong components of $D - X_n $, i.e. on the vertices of $D/P_n$ not in $X_n$ and that the intervals in $[0,1]$ on which $\alpha_n$ is constant have length less or equal to $\frac{1}{n}$. Now, the universal property of the inverse limit gives an injective topological path $\alpha$. It traverses every vertex of $D$ as the $X_n$ contain every vertex eventually, so  $\alpha$ is a Hamilton path in $|D|$. To show that $\alpha$ respects the normal order of $T$, consider two vertices $t \trianglelefteq_T t'$ and choose $n \in \N$ so that $t,t' \in X_n$. Then $t \trianglelefteq_{T_n} t'$ and $\alpha_n$ traverses $t$ before $t'$. As the projection of $\alpha$ to $D / P_n$ gives $\alpha_n$, we have that $\alpha$ traverses $t$ before $t'$. Conversely, if $\alpha$ traverses $t$ before $t'$ a we have that $\alpha_n$ traverses $t$ before $t'$ in $D/P_n$ for $n \in \N$ with $t,t' \in X_n$. Hence, we have $t \trianglelefteq_{T_n} t'$ and  as the normal order of $T$ induces the normal order of $T_n$, we have that $t \trianglelefteq_T t'$.
\end{proof}
\newpage

\begin{figure}[h]
    \centering
    \begin{tikzpicture}[scale=0.8]
    \begin{scope}[xscale=-1,yscale=1]
    \tikzset{edge/.style = {->,> = latex'}}
    
    

     
    
    \draw[fill,black] (-3,6.2) circle (.05);
    \draw[fill,black] (0,6.2) circle (.05);
    \draw[fill,black] (3,6.2) circle (.05);
    
    \node at (-2.6,6.4) {$\omega_1$};
    \node at (.4,6.4) {$\omega_2$};
    \node at (3.4,6.4) {$\omega_3$};
    
    
    \node at (-3,5.8) {$\vdots$};
    \node at (0,5.8) {$\vdots$};
    \node at (3,5.8) {$\vdots$};
    
    \node at (-2.7,4.6) {\footnotesize$R_1$};
    \node at (0.3,4.6) {\footnotesize$R_2$};
    \node at (3.3,4.6) {\footnotesize$R_3$};

     \foreach \y in {0,1,2,3} {
     \draw[fill,black] (0,\y+1) circle (.05);
     
     }
     
     \foreach \y in {1,2,...,3}  {
     \draw[edge,thick] (0,\y) to (0,\y+1);
     }

     \draw[edge,thick] (-3,4) to (-3,5.3);
     \draw[edge,thick] (0,4) to (0,5.3);
     \draw[edge,thick] (3,4) to (3,5.3);

     
     \foreach \y in {0,1,2,3} {
     \draw[fill,black] (3,1+\y) circle (.05);
     
     }
     
      \foreach \y in {1,2,...,3}  {
     \draw[edge, thick] (3,\y) to (3,\y+1);
     }
     
     
     \foreach \y in {0,1,2,3} {
     \draw[fill,black] (-3,1+\y) circle (.05);
     
     }
     
      \foreach \y in {1,2,...,3}  {
     \draw[edge, thick] (-3,\y) to (-3,\y+1);
     }

  \foreach \y in {1,2} {
  \draw[edge, gray] (-3,\y+2) to[bend right] (-3,1);
   \draw[edge, gray] (3,\y+2) to[bend left] (3,1);
   }

\draw[edge,thick] (-3,1) to (0,1);
\draw[edge,thick] (0,1) to (3,1);

    \foreach \y in {2,3,4} {
    \draw[edge, gray] (-3,1) to[bend left] (0,\y);
    \draw[edge, gray] (0,1) to[bend left] (3,\y);
    }

       \draw[edge, gray, dashed] (-3, 1 ) to[out=80, in=-150] (0,6.2);
    \draw[edge, gray, dashed] (0, 1 ) to[out=80, in=-150] (3,6.2);
    
    
 \draw[edge, gray, dashed] (-3,6.2) to[bend right] (-3,1);
 \draw[edge, gray, dashed] (3,6.2) to[bend left] (3,1);
 
\end{scope}
 
 \end{tikzpicture}
\caption{A countable solid tournament together with a normal spanning arborescence (black edges). All grey edges along a branch are oriented from top to bottom, and all edges between any two branches are oriented from right to left.}\label{fig: counterexample}
\end{figure}
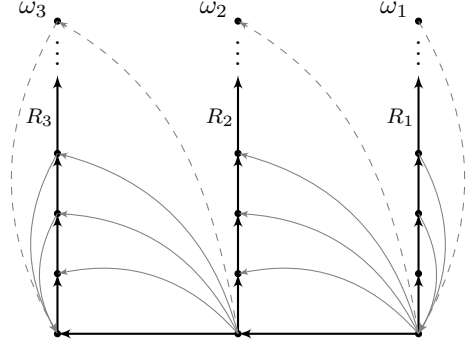

\begin{example}\label{example: redei best possible}
The tournament $D$ in Figure~\ref{fig: counterexample} satisfies the following properties:
\begin{enumerate}
    \item There is no spanning ray or double ray in $D$.
    \item There is no Hamilton path in the space formed by $D$ and its end compactification of the underlying undirected graph.
     \item There is no Hamilton path in $|D|$ that avoids all inner points of limit edges of $D$.
\end{enumerate}
\end{example}
\begin{proof}
To (i): There are no edges into $V(R_1)$ and there are no edges out of $V(R_3)$; hence, any ray or double ray that contains a vertex of $R_1$ and a vertex of $R_3$ contains only finitely many vertices of $R_2$.

To (ii): The underlying undirected graph of $D$ is a clique; hence, it has exactly one end $\omega$. Moreover, its end compactification  by  the ends of the underlying undirected graph coincides with its one-point compactification $D^*$. So the image of any Hamilton path in $D^*$ defines either a spanning (reverse) ray or the disjoint union of a ray and a reverse ray containing together every vertex of $D$. A similar argument as in (i) shows that there are no such (reverse) rays in $D$.

To (iii): Suppose for a contradiction that $D$ has a Hamilton path $\alpha$ that avoids all inner points of limit edges of $D$. Then one of the ends $\omega_i$ is not an endpoint of $\alpha$. So the image of $\alpha$ contains a ray  that represents $\omega_i$ and a reverse ray that contains infinitely many vertices of $R_i$. However, $D$ contains no such ray and reverse ray. 
\end{proof}

\section{Hamilton circles}\label{section: Camion's theorem}

\noindent In this section, we prove Theorem~\ref{mainthm: Camion}. For this we need some definitions and two Lemmas. The reverse subrays of a reverse ray are its \emph{tails}.  A reverse ray $R$ in a digraph $D$ \emph{represents} an end $\omega$ if there is a solid ray $R'$ in $D$ that represents $\omega$ such that $R$ and $R'$ have a tail in the same strong component of $D-X$ for every finite vertex set $X\subseteq V(D)$. It is straightforward to show that a reverse ray $R$ that represents an end $\omega$ defines a topological path from $\omega$ to the first vertex of $R$. For a (reverse) ray $R=v_1,v_2,\ldots$ the subpaths of the form $v_1,\ldots,v_n$ are the \emph{finite initial segments} of $R$. For a double ray $W=\ldots,w_{-1},w_0,w_{1},\ldots$ we denote by $W_{n<}$ the ray $w_{n+1},w_{n+2},\ldots$ and by $W_{<n}$ the reverse ray $\ldots,w_{n-2},w_{n-1}$.

We say that a vertex $v$ of $D$ can be \emph{inserted} into a (reverse) ray $R=v_1,v_2,\ldots$ if there is a path $P$ that starts (ends) at $v$ such that $v_1,\ldots,v_{i-1},P,v_{i},\ldots$ is a (reverse) ray in $D$, for some $i \in \N$. Similarly, we say that a vertex $v$ of $D$ can be \emph{inserted} into a double ray $W= \ldots ,w_{-1},w_0,w_1,\ldots$ if there is a path $P$ that contains $v$ such that $\ldots,w_{i-1},P,w_{i},\ldots$ is a double ray in $D$, for some $i \in \Z$.  A quick case distinction shows:

\begin{lemma}\label{lem: insertin}
Let $D$ be a strongly connected tournament. Any given vertex can be inserted into any given (reverse) ray.  Any given vertex with finite in- or finite  out-degree can be inserted into any given double ray.\qed
\end{lemma}

There is a natural partial order on the set of ends of a digraph $D$. For two ends $\omega , \eta \in \Omega(D)$ write $\omega \leq_\Omega  \eta$ if  there are rays $R_\omega$ and $R_\eta$  that  represent $\omega$ and~$\eta$ respectively, such that there are infinitely many  disjoint paths from $R_\omega$ to $R_\eta$. This gives a (well-defined) partial order on $\Omega(D)$. If $(\omega ,\eta)$ is a limit edge of $D$ then clearly $\omega \leq_\Omega \eta$; the converse is false in general. If $D$ is a tournament then any two ends of $D$ are comparable, so $\leq_\Omega$ gives a total order on $\Omega(D)$.

\begin{lemma}\label{lemma: end space tournament greates element}
Let $D$ be any countable solid tournament, then $\Omega(D)$ has a greatest and a least  element. 
\end{lemma}
\begin{proof}

First note that $\Omega(D)$ is non-empty; it is straightforward to construct a ray in $D$ since the deletion of any finite vertex set leaves only finitely many strong components and every ray in a solid digraph is solid. We show that $\Omega(D)$ has a greatest element; the proof for the least element is analogue.

Fix an enumeration of $V(D)$ and let $X_n$ denote the set of the first $n$ vertices. We have that $|D|~\cong~\varprojlim(D/ P_{n})_{n \in \N } $ by Corollary~\ref{cor: modD inverse limit of deletion minors}. Now, consider the strong components of $D-X_n$. We may view them as partially ordered by $C_1\le C_2$ if there is a path in $D-X_n$ from $C_1$ to $C_2$. As $D$ is a tournament, this gives a total order on the strong components of $D - X_n$. Hence, for every $X_n$ there is a greatest strong component $C_n$ of $D-X_n$ with regard to the aforementioned order of strong components. This strong component is a vertex in $D / P_{n}$, and this choice of vertices is compatible in the sense that $C_n$ includes $C_{n+1}$ as a subset. So this choice of vertices  gives a point in the inverse limit, which in turn corresponds to a point $\omega$ in $|D| $. It is straightforward to show that this point $\omega$ is an end of $D$ and we claim that it is the greatest element of $\Omega(D)$. Indeed, for any other end $\eta \in \Omega(D)$ there is a $X_n$ that separates $\omega$ and $\eta$. Now, $C( X_n, \omega )$ and $C( X_n, \eta )$ are distinct strong components of $D-X_n$, and by the choice of $\omega$ we have that $C( X_n, \omega )$  is greater than $C( X_n, \eta )$. Consequently, there are only finitely many disjoint paths from a ray representing $\omega$ to a ray representing $\eta$.
\end{proof}

\begin{mainresult}\label{mainthm: Camion}
Every countable strongly connected solid tournament has a  Hamilton  circle. 
\end{mainresult}
\begin{proof}

First note that for every vertex $v \in V(D)$ and any end $\omega \in \Omega(D)$ the tournament $D$ has a limit edge from $v$ to $\omega$ or vice versa. Indeed, for a ray $R$ that represents $\omega$ the vertex $v$ sends infinitely many edges to $V(R)$ or receives infinitely many edges from $V(R)$. Furthermore, for a vertex $v \in V(D)$ there is a limit edge from some end of $D$ to $v$ if and only if $v$ has infinite in-degree, and  there is a limit edge from $v$ to some end of $D$  if and only if $v$ has infinite  out-degree.  We split the proof into two main cases:\newline

\noindent \textbf{First case:} The tournament $D$ has only one end $\omega$. In this case, first suppose that every vertex of $D$ has infinite  in- and out-degree. Then it is straightforward to construct a spanning ray $R$. The first vertex $v$ of $R$ receives a limit edge from $\omega$, and following first $R$ to $\omega$ and then the limit edge $(\omega,v)$ yields the desired Hamilton circle in $|D|$. We remark that in this situation it is also possible to construct a spanning reverse ray or a spanning double ray to obtain a Hamilton circle in $|D|$. 

Second, suppose that there is a vertex, $v$ say, of $D$ that has finite in- or finite out-degree. We discuss the case where $v$ has finite in-degree, the other case follows by considering the reverse of $D$. Fix a normal spanning arborescence $T$ of $D$ rooted at $v$, Lemma~\ref{lem: countabledigraphs have NSA}, and apply Theorem~\ref{mainthm: Redei} to $D$ and $T$ to obtain a Hamilton path $\alpha$ in $|D|$ that starts at $v$. As $D$ is solid and one-ended, $T$ has exactly one ray $R_\omega$ that starts at $v$, Lemma~\ref{lem: NST endfaith}. All vertices of $R_\omega$ are traversed by $\alpha$ before $\omega$, in particular the vertices that are traversed by $\alpha$ before $\omega$ form a ray $R$, in the order in which they are traversed by $\alpha$. Conversely, the vertices that are traversed by $\alpha$ after $\omega$ form a reverse path or a reverse ray $\Rv$. Furthermore, every vertex $v'$ of $\Rv$ has finite  out-degree as there are only finitely many vertices greater than $v'$ in the normal order. If $\Rv$ is a reverse path we are done by Lemma~\ref{lem: insertin}, as we can insert all the vertices of $\Rv$ into $R_\omega$ one after another to obtain a spanning ray and then use a limit edge from $\omega$ to its start vertex to obtain the desired Hamilton circle in $|D|$. If $\Rv$ is a reverse ray there is an edge from some vertex of $\Rv$ to $v$, since $v$ has finite out-degree. Consequently, we find a double ray $W$ that contains $R$ and all but finitely many vertices of $\Rv$. By Lemma~\ref{lem: insertin}, any vertex of $\Rv$ not yet contained in $W$ can be inserted into $W$ to obtain a spanning double ray of $D$. As $D$ has only one end, a spanning double ray naturally defines a Hamilton circle in $|D|$.\newline

\noindent \textbf{Second case:} The tournament $D$ has more than one end. For the rest of the proof denote by $\omega_*$ the least and by $\omega^*$ the greatest end of $D$, Lemma~\ref{lemma: end space tournament greates element}. Our first goal is to find a double ray $W$ such that its finite subpaths separate every vertex and every other end from $\omega_*$ and $\omega^*$ respectively, and such that its subrays represent $\omega_*$ and its reverse subrays represent $\omega^*$. In order to find such a double ray, fix a normal spanning arborescence $T$ of $D$. Then there is exactly one ray $R_{\omega_*}$ in $T$ that starts at the root of $T$ and represents $\omega_*$ in $D$, Lemma~\ref{lem: NST endfaith}. This ray $R_{\omega_*}$ has the property that its finite initial segments separate $\omega_*$ from every vertex and every other end eventually, Lemma~\ref{lemma: separation props of NSA}. Now, consider $\Dv$ the reverse of $D$ and fix a normal spanning arborescence $\Tv$ of $\Dv$. The ends of $D$ and $\Dv$ are in a one-to-one correspondence in that the reverse of every ray in $D$ represents an end in $\Dv$.  Let $R_{\omega^*}$ be the unique ray in $\Tv$ that starts at the root and represents the least end of $\Dv$, Lemma~\ref{lem: NST endfaith}. Then the reverse $\Rv_{\omega^*}$ of $R_{\omega^*}$  is a reverse ray in $D$ that represents $\omega^*$ and its finite initial segments separate $\omega^*$ from every vertex and every other end eventually, Lemma~\ref{lemma: separation props of NSA}. As $\omega_* \neq \omega^*$, we have that $R_{\omega_*}$ and $\Rv_{\omega^*}$ have only finitely many vertices in common. Consequently, there is a double ray $W'$ that contains a tail of $R_{\omega_*}$ and a tail of $\Rv_{\omega^*}$. Let $S$ be the set of all the vertices of $R_{\omega_*}$ and $\Rv_{\omega^*}$ not contained in $W'$. If we can insert all those vertices of $S$ into $W'$ that are not separated from $\omega_{*}$ or $\omega^*$ by any finite subpath of $W'$ then we obtain our desired double ray. So let $s \in S$ be a vertex that is not separated from $\omega_{*}$ or $\omega^*$ by any finite subpath of $W'$. If $s$ can be separated from $\omega_*$ but not from $\omega^*$  by a finite subpath of $W'$, or vice versa, then it is straightforward to insert $s$ into $W'$, Lemma~\ref{lem: insertin}. So we may assume that $s$ cannot be separated from $\omega_*$ and from $\omega^*$ by  any finite subpath of $W'$. Furthermore, we may assume that for $W'=\ldots,w_{-1},w_{0},w_1,\ldots $ we have that $s$ receives an edge from $w_{n}$ for $0\leq n$ and sends an edge to all the other vertices of $W'$, otherwise $s$ can clearly be inserted into $W'$. As $\omega_* <_\Omega \omega^*$, there is a finite subpath $W_N= w_{-N}, \ldots, w_N$ of $W'$, for $2 \leq N$, such that there is no edge from $W_{<N}$ to $W_{N<}$. Now, there is a non-trivial path $P$ in $C(W_N,\omega_*)= C(W_N,\omega^*)$ from $W_{<N}$  to $W_{N<}$. If $P$ contains $s$ it can be inserted into  $W_{N<}$ via a subpath of $P$ (and into $W_{N<}$). If $s$ is not contained in $P$ consider any inner vertex $v$ of $P$, then $s$ can be inserted into  $W_{N<}$ or $W_{<N}$ depending on whether $D$ contains the edge $(s,v)$ or the edge $(v,s)$. Hence, we obtain a double ray $W$ such that its finite subpaths separate every vertex and every other end from $\omega_*$ and $\omega^*$ respectively and such that its subrays represent $\omega_*$ and its reverse subrays represent $\omega^*$.

Our next goal is to show that there is even a double ray with the defining properties of $W$ that contains all vertices which send a limit edge to $\omega_*$ and all vertices that receive a limit edge from $\omega^*$. We will show this claim by inserting all these vertices, not yet contained in $W$,  one after the other into $W$ in such a way that the limit is still a double ray. 

Denote by $S_*$ all vertices not in $W$ which send a limit edge to $\omega_*$, and by $S^*$ all vertices not in $W$ which receive a limit edge from $\omega^*$. Note that $S_*$ and $S^*$ are disjoint, as there is a finite subpath of $W$ that separates $\omega_*$ and $\omega^*$. First consider $S_*$ and choose $N \in \N$ so that $W_N$ separates $\omega_*$ and $\omega^*$. It is straightforward to check that all vertices in $S_*$ that are separated from $\omega_*$ by $W_N$ can be inserted into $W$ without changing $W_{N<}$ or $W_{<N}$. Now, for any other vertex $ s \in S_*$ there is a smallest $ n(x) \in \N $ such that $W_{n(s)}$ separates $s$ from $\omega_*$. Again, it is straightforward to check that all vertices in $S_*$ with index $n(s)$ can be inserted into $W$ by a path from $w_{n(s)}$ to $w_{n(s)+1}$. An analogue technique shows that all vertices in $S^*$ can be inserted into $W$. As we substituted only edges of $W$ by a path at most once, we end up in the limit step with a double ray.

So let us assume that $W$ additionally has the property to contain  $S_* \cup S^*$. Our final goal is to find an injective topological path $\alpha$ from $\omega_*$ to $\omega^*$ that contains precisely the vertices not in $W$. Having $\alpha$ at hand, the desired Hamilton circle in $|D|$ is obtained by first following $\alpha$ and then following $W$.

Consider the strong components of $D- W$, for any such strong component there is a finite subpath $W_n$ of $W$ such that $C$ is a strong component of $D-W_n$. Indeed, for every $v\in C$ there is an $n(v) \in \N$ such that $W_{n(v)}$ separates $v$ from $\omega_*$ and $\omega^*$ and this $n(v)$ has to be the same for any two vertices in $C$. Moreover, these strong components are totally ordered in that every vertex of $C$ sends an edge to any vertex of $C'$, or vice versa, for any two strong components of $D-W$. For all strong components of $D-W$ fix a Hamilton path $\alpha_C$ in $|C|$ (or in $C$ if it is finite). Now, all these Hamilton paths can be linked up to the desired injective topological path $\alpha$, see Figure~\ref{fig:secproof}. Indeed, if $C$ is the predecessor of $C'$ in the aforementioned order of strong components of $D-W$, then there is an edge or  limit edge from the endpoint of $\alpha_C$ to the starting point of $\alpha_{C'}$. Moreover, if there is a least element, $C_*$ say, of the strong components of $D-W$, then $\omega_*$ sends a limit edge to every vertex of $C_*$. Similarly, if there is a greatest element, $C^*$ say, then every vertex of $C^*$ sends a limit edge to $\omega^*$. 

Conversely, if there is no greatest element, then the strong components of $D-X$ converge to $\omega^*$ in that traversing the $\alpha_C$ one after the other in their total order yields an injective continuous path that ends at $\omega^*$. Similarly,  if there is no least element, then the strong components of $D-X$ converge to $\omega_*$ in that traversing the $\alpha_C$ one after the other in their inverted total order yields an injective continuous path that starts at $\omega_*$. Note that the open sets $\hat{C}_\varepsilon(W_n,\omega_*)$ and $\hat{C}_\varepsilon(W_n,\omega^*)$ form a neighbourhood base for $\omega_*$ and $\omega^*$, respectively. This topological path traverses all the vertices of $D-W$ as $W$ contains $S_* \cup S^*$. We remark that if there are no strong components of $D-W$, i.e. $W$ is spanning then we obtain a directed topological path from $\omega_*$ to $\omega^*$ that avoids $W$ by the limit edge $(\omega_*,\omega^*)$.
\end{proof}

\begin{figure}[ht]
    \centering
    \begin{tikzpicture}[scale=.8]
    \begin{scope}[xscale=-1,yscale=1]
    \tikzset{edge/.style = {->,> = latex'}}
    
    \draw[fill,black] (0,0) circle (.05);
    
    \draw[fill,black] (-2,0.3) circle (.05);
    \draw[fill,black] (2,0.3) circle (.05);
    
    \draw[fill,black] (-3.5,1) circle (.05);
    \draw[fill,black] (3.5,1) circle (.05);
    
    \draw[fill,black] (-4.5,2) circle (.05);
    \draw[fill,black] (4.5,2) circle (.05);
    


 \draw[edge, thick] (2,0.3) to[in=0, out=-165] (0,0);
 \draw[edge, thick] (0,0) to[in=-15, out=180] (-2,0.3);

 \draw[edge, thick] (3.5,1) to[in=15, out=-145] (2,0.3);
  \draw[edge, thick] (-2,0.3) to[out=-15, in=-35] (-3.5,1);
  
   \draw[edge, thick] (4.5,2) to[in=38, out=-130] (3.5,1);
   
    \draw[edge, thick] (-3.5,1)  to[in=-50, out=142] (-4.5,2);

\draw[edge, thick]  (5,3.3) to[out=-90, in=50](4.5,2);
\draw[edge, thick]  (-4.5,2) to[in=-90, out=130] (-5,3.3);
 
    \draw[fill,black] (-5,4) circle (.05);
     \draw[fill,black] (5,4) circle (.05);
     
    \node at (-5,3.7) {$\vdots$};
     \node at (5,3.7) {$\vdots$};





\begin{scope}[shift={(-3.5,0.5)}]
\draw (-1.8,3.5) to[out=-90, in=180] (0,0);
\draw (1.8,3.5) to[out=-90, in=0] (0,0);
\end{scope}

\begin{scope}[shift={(3.5,0.5)}]
\draw (-1.8,3.5) to[out=-90, in=180] (0,0);
\draw (1.8,3.5) to[out=-90, in=0] (0,0);
\end{scope}

\begin{scope}[shift={(-4.2,1.6)}]
\draw (-1,2.5) to[out=-90, in=180] (0,0);
\draw (1,2.5) to[out=-90, in=0] (0,0);
\end{scope}

\begin{scope}[shift={(4.2,1.6)}]
\draw (-1,2.5) to[out=-90, in=180] (0,0);
\draw (1,2.5) to[out=-90, in=0] (0,0);
\end{scope}


\draw (-1,1.5) circle (.5);
\draw (1,1.5) circle (.5);

\draw (-2.5,2.5) circle (.4);
\draw (2.5,2.5) circle (.4);

\draw (-3.8,3.5) circle (.3);
\draw (3.8,3.5) circle (.3);

 \draw[edge, thick, dashed] (-3.8,3.2) to[out=-90,in=180] (-2.9,2.5);
 
 \draw[edge, thick, dashed] (2.9,2.5) to[out=0, in=-90] (3.8,3.2);
 
 \draw[edge, thick, dashed] (-2.1,2.5) to[bend left] (-1,2);
 \draw[edge, thick, dashed] (1,2) to[out=70,in=180] (2.1,2.5);
 
 \draw[edge, thick, dashed] (-1,1) to[out=-45,in=-135]  (1,1);

    
\node at (-5+0.35,4.2) {$\omega_*$};
\node at (5-0.4,4.2) {$\omega^*$};

\node at (0,-0.3) {$w_0$};
\node at (-2,0) {$w_1$};
\node at (2.2,0) {$w_{-1}$};

\node at (5.3,0.5) {\footnotesize{$C(W_1,\omega^*)$}};

\node at (-1.7,1.1) {$C$};
\node at (1.7,1.1) {$C'$};

\end{scope}

\end{tikzpicture}
\caption{Strong components of $D-W$ are indicated as circles. Strong components of the form $C(W_n,\omega_*)$ or $C(W_n,\omega^*)$ are indicated as parabolas, which might contain strong components of $D-W$ not yet separated by $W_n$. }\label{fig:secproof}
\end{figure}
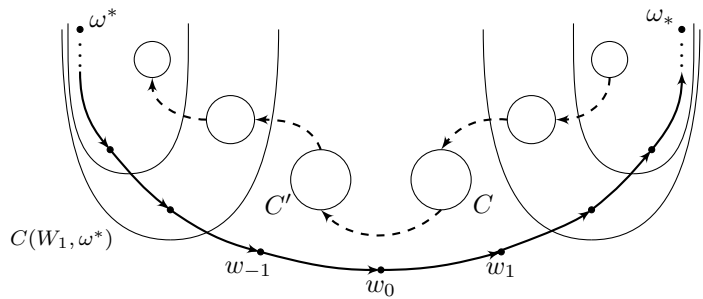

\bibliographystyle{plain}
\bibliography{reference}

\end{document}